\documentclass{amsart}
\usepackage[spanish]{babel}
\usepackage{amssymb}	
\usepackage{amsthm}
\usepackage{graphicx}
\usepackage{tikz}

\theoremstyle{plain}
\newtheorem{definition}{Definici\'on}[section]
\newtheorem{question}[definition]{Pregunta}
\newtheorem{theorem}[definition]{Teorema}
\newtheorem{fact}[definition]{Observaci\'on}

\newtheorem{lemma}[definition]{Lema}

\newtheorem{example}[definition]{Ejemplo}

\title{Un teorema de Ramsey para los enteros}

\author{Andr\'es Eduardo Caicedo}
\address{
Mathematical Reviews\\
  416 Fourth Street\\ 
  Ann Arbor, MI 48103-4820\\ 
  USA
}
\email{aec@ams.org}
\urladdr{http://www-personal.umich.edu/~caicedo/}

\begin{document}
\maketitle

\section{El teorema de Ramsey} \label{section:intro}

En todo grupo de diez personas siempre podemos encontrar, o bien cuatro que se han hablado  
mutuamente, o bien tres que no se han hablado entre s\'\i.

\'Este es un caso particular del teorema de Ramsey, un resultado fundamental en la combinatoria de 
grafos.

\begin{theorem}[Ramsey {\cite[Theorem B]{MR1576401}}] \label{thm:ramseyf}
Para todos los enteros positivos $n,m$ podemos encontrar un n\'umero $N$ tal que en todo grafo 
con $N$ v\'ertices podemos encontrar, o bien una copia del grafo completo de $n$ v\'ertices, o bien 
$m$ v\'ertices sin arcos entre s\'\i.
\end{theorem}
 
Frank Ramsey \cite{MR1576401} estableci\'o (una versi\'on general de) este resultado, con el que 
se  inaugur\'o la teor\'\i a que lleva su nombre, con el fin de estudiar la consistencia de ciertas 
f\'ormulas  de la l\'ogica de primer orden. Desde entonces se han encontrado much\'\i simas 
aplicaciones en diversas \'areas, as\'\i\ como generalizaciones y variantes. 

Ramsey present\'o dos versiones de su teorema, la versi\'on finita que enunciamos arriba, y una 
infinita. En esta nota estamos interesados en la versi\'on infinita. Primero, algo de notaci\'on: dado 
un conjunto $X$, el grafo completo $K_X$ tiene a $X$ como conjunto de v\'ertices, con arcos 
(aristas) entre todo par de v\'ertices distintos. El grafo vac\'\i o $\bar K_X$ tiene a $X$ como 
conjunto de v\'ertices, y no tiene arcos. Escribimos $K_n$ para denotar el grafo completo en un 
conjunto de $n$ v\'ertices. En esta nota, los subgrafos son siempre \emph{inducidos}, es decir que 
si $\mathcal G=(V_G,A_G)$ es un grafo, $\mathcal H=(V_H,A_H)$ es un subgrafo si y s\'olo si 
$V_H\subseteq V_G$ and $A_H=\{\{v,w\}\in A_G:v,w\in V_H\}$. En particular, que $\mathcal G$ 
contiene una copia de un grafo $\mathcal J$ quiere decir que la contiene como subgrafo inducido.  

\begin{theorem}[Ramsey {\cite[Theorem A]{MR1576401}}] \label{thm:ramseyinf}
Dado un conjunto infinito $X$, si los arcos de $K_X$ se particionan en dos clases, hay un 
subconjunto infinito $H\subseteq X$ tal que todos los arcos de $K_H$ est\'an en la misma clase.
\end{theorem}

De manera equivalente, el teorema dice que para todo grafo $\mathcal G$ con conjunto de v\'ertices 
$X$ hay un subconjunto infinito de v\'ertices $H$ tal que $\mathcal G$ contiene o bien al grafo 
completo $K_H$ o bien al grafo vac\'\i o $\bar K_H$. Tambi\'en podemos formular el resultado 
diciendo que si coloreamos los arcos de $K_X$ rojos o azules, hay un $H$ infinito tal que o bien 
$K_H$ es rojo, o bien es azul. (En lo que sigue, todos los coloreos son siempre de los arcos de un 
grafo, en rojo y azul.)

\begin{definition} \label{def:count}
Un conjunto infinito es \emph{contable} si y s\'olo si est\'a en biyecci\'on con (es decir, tiene el 
mismo tamaño que) $\mathbb N$, el conjunto de los n\'umeros naturales. 
\end{definition}

En cierto sentido, el teorema \ref{thm:ramseyinf} es un resultado acerca del orden de los n\'umeros 
naturales. El objetivo de esta nota es demostrar una versi\'on del resultado acerca del orden de 
los enteros.  

\section{El c\'alculo de particiones}

Es conveniente introducir algo de notaci\'on antes de proseguir. 

\begin{definition} \label{def:arrowf}
Dados enteros positivos $N,a,b$, escribimos 
$$ N\to(a,b)^2 $$ 
(lo que, tristemente, se lee `$N$ flecha $a,b$ super 2') para denotar que cualquier 
coloreo de $K_N$ contiene, o bien una copia roja de $K_a$, o bien una copia azul de $K_b$.  

Escribimos $N \nrightarrow (a,b)^2$ para indicar lo contrario.

El estudio de esta relaci\'on flecha y sus variantes es el \emph{c\'alculo de particiones}.
\end{definition}

El super\'\i ndice 2 en la definici\'on \ref{def:arrowf} indica que \'este es un caso particular (el caso 
bidimensional) de  notaci\'on m\'as general; m\'as adelante discutiremos el caso de super\'\i ndice 1 
(el principio de casillas). La notaci\'on se debe a Rado, quien la introdujo en su art\'\i culo 
\cite{MR0058687} conjunto con Erd\H os, considerando m\'as generalmente la versi\'on 
\emph{cardinal}, donde $N,a,b$ pueden ser tamaños de conjuntos infinitos. La notaci\'on es algo 
intimidante cuando uno la encuentra por primera vez, pero es muy conveniente en esta \'area.

\begin{example} \label{ex:fin}
Que ${5 \nrightarrow (3,3)^2}$ significa que hay un coloreo de $K_5$  tal que no hay tri\'angulos con 
los tres arcos del mismo color (el tri\'angulo es \emph{monocrom\'atico}). Por ejemplo, el coloreo 
ilustrado en la figura de la izquierda.

Por otro lado, que $6\to(3,3)^2$ significa que en cualquier coloreo de $K_6$ hay un tri\'angulo 
monocrom\'atico: cualquier v\'ertice $v$ (en la figura de la derecha, el v\'ertice verde) est\'a unido a 
por lo menos otros tres v\'ertices por arcos del mismo color $c$. Si cualquiera de los arcos entre 
estos v\'ertices es tambi\'en de color $c$, estos dos v\'ertices y $v$ forman un tri\'angulo de color 
$c$. En caso contrario, estos tres v\'ertices forman un tri\'angulo del otro color.

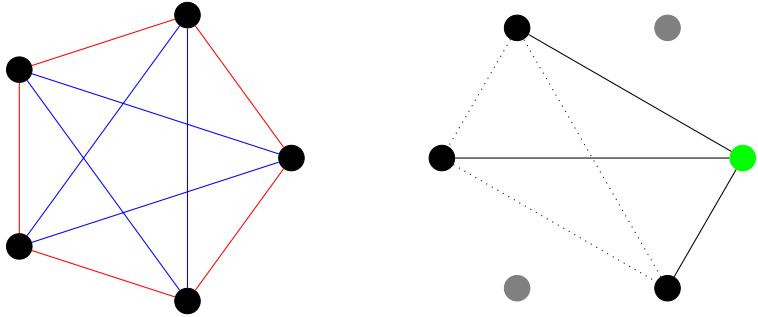
\begin{figure}[ht] 
\centering
\begin{tikzpicture}[%
  every node/.style={draw,fill=white,circle,minimum size=3pt},node distance=1cm]
  \node[color=black]  (one) at (-1,0) {};
  \node[color=black]  (two)  at (2*0.309-3,2*0.951) {};
  \node[color=black]  (three) at (2*-0.809-3,2*0.588) {};
  \node[color=black]  (four) at (2*-0.809-3,2*-0.588) {};
  \node[color=black]  (five)  at (2*0.309-3,2*-0.951) {};
  \node[color=green]  (uno) at (5,0){};
  \node[color=gray] (cinco) at (2*0.5+3,2*0.866) {};
  \node[color=black]  (dos)  at (2*-0.5+3,2*0.866) {};
  \node[color=black]  (tres) at (2*-1+3,0) {};
  \node[color=gray] (seis) at (2*-0.5+3,2*-0.866) {};
  \node[color=black]  (cuatro) at (2*0.5+3,2*-0.866) {};
  \draw [red](one) -- (two) -- (three) -- (four) -- (five) -- (one);
  \draw [blue](one) -- (three) -- (five) -- (two) -- (four) -- (one);
  \draw (tres) -- (uno) -- (dos);
  \draw (uno) -- (cuatro);
  \draw [dotted](dos) -- (tres) -- (cuatro) -- (dos);
\end{tikzpicture}
\end{figure}
\end{example}

Se sigue que, para todo $N\ge6$, $N\to(3,3)^2$, y que no podemos reemplazar el 6 por un 
n\'umero menor. Decimos entonces que 6 es el \emph{n\'umero de Ramsey} de 3,3.

\begin{definition}
El n\'umero de Ramsey de $a,b$, $R(a,b)$, es el menor $N$ tal que $N\to(a,b)^2$. 
\end{definition}

Que $N$ existe es precisamente el contenido del teorema \ref{thm:ramseyf}. En esta notaci\'on, el 
ejemplo con el que comenzamos la nota, dice que $10\to(4,3)^2$, de modo que $R(4,3)\le 10$. 
De hecho, $R(4,3)=9$.

En lo que sigue, estudiamos la variante de la relaci\'on flecha donde, en vez de n\'umeros, 
consideramos \'ordenes lineales.

\section{\'Ordenes lineales}

Recu\'erdese que una biyecci\'on entre dos conjuntos $A$ y $B$ es una funci\'on 
${\varphi\!:A\to B}$ que es inyectiva (puntos distintos de $A$ son enviados por $\varphi$ a puntos 
distintos de $B$) y sobreyectiva (todo punto de $B$ es la imagen bajo $\varphi$ de un punto de 
$A$). Intuitivamente, podemos pensar en $A$ y $B$ como dos idiomas y en $\varphi$ como un 
diccionario que traduce de manera perfecta entre los dos.

Si $X$ es contable, por ejemplo $X=\mathbb N$,  todo $H$ como en el teorema \ref{thm:ramseyinf} 
tiene el mismo tamaño que $X$. Denotando por $\aleph_0$ (`alef-0') el tama\~no de los n\'umeros 
naturales, tenemos que $\aleph_0\to(\aleph_0,\aleph_0)^2$. Sin embargo, y esta observaci\'on es la 
que nos lleva a todo lo que sigue, si imponemos una estructura adicional en $X$, por ejemplo, si en 
vez de mirar s\'olo al conjunto $X$, consideramos en cambio un orden lineal $(X,<)$ (simplemente 
diremos un orden), el teorema \ref{thm:ramseyinf} no garantiza que la restricci\'on a $H$ de este 
orden sea no trivial, en el sentido que explicamos siguiendo la observaci\'on \ref{fact:ww*}. 

Los \'ordenes finitos son muy simples: para cada $n$, hay un \'unico orden de tamaño~$n$. Por 
ejemplo, para $n=4$, el orden es
 $$ \bullet\quad\bullet\quad\bullet\quad\bullet $$
(en t\'erminos t\'ecnicos, todos estos \'ordenes son isomorfos). La situaci\'on es completamente 
diferente cuando consideramos conjuntos infinitos: aunque $\mathbb N$, $\mathbb Z$ y 
$\mathbb Q$ son contables, como \'ordenes lineales lucen muy distintos. 

\begin{definition} \label{def:iso}
Decimos que dos \'ordenes $(L_1,<_1)$ y $(L_2,<_2)$ son \emph{isomorfos} o \emph{tienen el 
mismo tipo de orden} o, simplemente, \emph{tienen el mismo tipo} si y s\'olo si hay una biyecci\'on 
$\varphi\!:L_1\to L_2$  entre ellos que transforma un orden en el otro, es decir que para 
cualesquiera $x_1,x_2\in L_1$, tenemos que 
 $$ x_1<_1 x_2 \quad\iff\quad \varphi(x_1)<_2 \varphi(x_2). $$
\end{definition}

En t\'erminos t\'ecnicos, los tipos de orden son las clases de equivalencia de la relaci\'on de 
isomorfismo entre \'ordenes. En lo que sigue, para simplificar notaci\'on, a veces en vez de un 
orden $(L,<)$, escribimos su clase de equivalencia (su tipo). En particular, denotamos por 
$\omega$ (`omega') el tipo de los n\'umeros naturales, y por $\omega^*$ el tipo de los enteros 
negativos. Para $n\in\mathbb N$, es usual denotar por $n$ el tipo de cualquier orden lineal con 
$n$ elementos. Abusando el lenguaje, tambi\'en decimos que, por ejemplo, un orden es isomorfo 
a $\omega$ para indicar que es isomorfo a los n\'umeros naturales. 

En cierto sentido, los \'ordenes de tipo $\omega$ o de tipo $\omega^*$ son los \'ordenes infinitos 
m\'as simples:

Si $(X,<)$ es un orden y $Y\subseteq X$, decimos que la restricci\'on de $<$ a $Y$, es decir, el 
orden $(Y,<\,\upharpoonright Y\times Y)$, es un \emph{suborden} de $(X,<)$.

\begin{fact} \label{fact:ww*}
Todo orden infinito contiene un suborden isomorfo a $\omega$ o a $\omega^*$. 
\end{fact}

Se sigue que, aunque $X$ sea un orden muy complicado, el suborden~$H$ que obtenemos en el 
teorema \ref{thm:ramseyinf} puede en cambio ser simple (isomorfo a $\omega$ o a $\omega^*$). 
De hecho, es en general \emph{imposible} obtener una versi\'on del teorema \ref{thm:ramseyinf} en 
la que el orden de $H$ sea m\'as elaborado.

\begin{theorem}[Erd\H os-Rado {\cite[Theorem 19, Corollary]{MR81864}}] \label{thm:ER}
No existe un orden~$X$ tal que siempre que coloreemos $K_X$, o bien 
podemos encontrar una copia roja de $K_\omega$ o bien una copia azul de $K_{\omega^*}$.
\end{theorem}

\section{El c\'alculo de particiones para \'ordenes}

El teorema \ref{thm:ER} impone restricciones significativas a las posibles extensiones del teorema 
\ref{thm:ramseyinf}. Dados \'ordenes lineales $L,L_1,L_2$, escribamos 
$$ L\to(L_1,L_2)^2 $$ 
para denotar que cualquier coloreo de $K_L$ contiene, o bien una copia roja de $K_{L_1}$, o bien 
una copia azul de $K_{L_2}$.  La notaci\'on es una extensi\'on debida a Erd\H os y Rado 
\cite{MR81864}  de la versi\'on discutida antes. 

As\'\i\ como $\omega^*$ denota el orden `reverso' de $\omega$, en general dado un orden 
$L$, $L^*$ denota el orden en el mismo conjunto que $L$ pero donde $x<y$ en $L^*$ si y s\'olo 
si $x>y$ en $L$.

\begin{fact} \label{fact:obs}
Para cualesquiera \'ordenes lineales $L,L',L_1,L_1',L_2$ tenemos:
\begin{enumerate}
\item $L\to(L_1,0)^2$; si $L$ no es vac\'\i o, entonces $L\to(L_1,1)^2$;  $L\to(L,2)^2$.
\item Si $L_1$ y $L_2$ tienen por lo menos dos puntos y $L\to(L_1,L_2)^2$, entonces $L_1,L_2$ 
son sub\'ordenes de $L$.
\item $L\to(L_1,L_2)^2$ si y s\'olo si $L\to(L_2,L_1)^2$.
\item Si $L'$ es un orden que contiene a $L$ como suborden, $L_1'$ es un suborden de $L_1$, y 
$L\to(L_1,L_2)^2$, entonces tambi\'en $L'\to(L_1',L_2)^2$. 
\item Si $L\to (L_1,L_2)^2$, entonces $L^*\to(L_1^*,L_2^*)^2$.
\item Por el teorema \ref{thm:ramseyinf},
$\omega\to(\omega,\omega)^2$ y $\omega^*\to(\omega^*,\omega^*)^2$.
\item Por el teorema \ref{thm:ER}, $L \nrightarrow (\omega,\omega^*)^2$.
\end{enumerate}
\end{fact}

?`Qu\'e m\'as podemos decir para \'ordenes infinitos? 

Si queremos un teorema de la forma $L\to(L_1,L_2)^2$ con $L_1$ infinito, se sigue de la 
observaci\'on \ref{fact:obs} que, o bien $L_2$ es finito, o bien (1) ninguno de $L_1,L_2$ contiene 
una cadena descendente infinita (una copia de $\omega^*$), o (2) ninguno de los dos contiene una 
cadena ascendente infinita (una copia de $\omega$).  El segundo caso es el `reverso' del primero 
(y por tanto su estudio se reduce al estudio del primero, por 5. de la observaci\'on \ref{fact:obs}), y 
el primer caso, el c\'alculo de particiones para ordinales, se ha estudiado extensamente. 

\begin{definition} 
Dados \'ordenes $A$ y $B$, $A+B$ denota el orden que resulta de tener una copia de $A$ seguida
por una de $B$. Por $A\cdot B$ denotamos el orden que resulta de comenzar con $B$ y 
reemplazar cada punto con una copia de $A$.
\end{definition}

\begin{example}
\begin{itemize}
\item $\omega^*+\omega$ es el tipo de los enteros:
 $$ \cdots \quad \bullet\quad\bullet\,\quad\,\bullet\quad\bullet \quad \cdots. $$
\item $\omega+\omega^*$ es el tipo de la restricci\'on del orden de los reales al conjunto
\begin{gather*}
 \left\{1-\frac1{n+1} : n\in\mathbb N\right\} \cup \left\{1+\frac1{n+1} : n\in\mathbb N\right\}: \\ 
\bullet\quad\bullet\quad\bullet\quad\cdots \quad \cdots\quad\bullet\quad\bullet\quad\bullet. 
\end{gather*}
\item $\omega\cdot 2=\omega+\omega$ es el tipo del orden que resulta de tener dos copias de los 
n\'umeros naturales. 
\item $\omega^2=\omega\cdot\omega$ es el tipo de 
 $$ \left\{ n-\frac1{m+1} : m,n\in\mathbb N\right\}. $$
\end{itemize}
\end{example}

Ejemplos de resultados del c\'alculo de particiones de ordinales son la observaci\'on de que 
 $$ \omega\cdot 2+1\to(\omega+1,3)^2 $$
(y, de hecho, \'este es el n\'umero de Ramsey: $R(\omega+1,3)=\omega\cdot 2+1)$, y el siguiente 
teorema.

\begin{theorem}[Specker {\cite[Satz 1]{MR88454}}] \label{thm:specker}
Para todo $n\in\mathbb N$, $\omega^2\to(\omega^2,n)^2$. 
\end{theorem}

La observaci\'on \ref{fact:ww*} dice que todo orden infinito, contiene una cadena ascendente o una 
descendente. Por tanto, si $L_1$ es infinito, cualquier teorema de la forma 
 $$ L\to(L_1,L_2)^2 $$ 
no cubierto por el c\'alculo de particiones para ordinales (o su `reverso') requiere que $L_1$ 
contenga cadenas de ambos tipos y que $L_2$ sea finito. Los dos tipos de \'ordenes m\'as simples 
que contienen cadenas ascendentes y descendentes son $\omega^*+\omega$ y 
$\omega+\omega^*$, y (despu\'es de estudiar ambas opciones) resulta que la m\'as sencilla de 
estas dos posibilidades es cuando $L_1$ es de tipo $\omega^*+\omega$, for ejemplo 
$L_1=\mathbb Z$. Finalmente podemos ser expl\'\i citos: el objetivo de esta nota es el analisis 
exhaustivo de este caso, es decir, queremos establecer todos los teoremas  de la forma 
$L\to(\mathbb Z,n)^2$. 
 
Por la observaci\'on \ref{fact:obs}, $L\to(\mathbb Z,0)^2$ para cualquier $L$,  $L\to(\mathbb Z,1)^2$ 
para cualquier $L$ no vac\'\i o, y $L\to(\mathbb Z,2)^2$ si y s\'olo si $L$ contiene un suborden 
isomorfo a $\mathbb Z$. El caso en que $L_2$ es finito y tiene por lo menos tres elementos es 
m\'as complicado, y el resto de esta nota se dedica a estudiarlo. Lo que debemos hacer es 
caracterizar, para cada $n>2$, todos los \'ordenes lineales $L$ tales que $L\to(\mathbb Z,n)^2$.

\section{Un resultado de casillas} \label{sec:casillas}

Para llevar a cabo este an\'alisis debemos resolver antes una pregunta aparentemente m\'as 
sencilla:

\begin{question} \label{q:Z}
?`Qu\'e tan largo ha de ser un orden $L$ para asegurar que, para cualquier partici\'on $L=A\cup B$ 
de $L$ en dos piezas, al menos una de las piezas contiene una copia de los enteros?
\end{question}

Esta es una pregunta t\'\i pica del estudio del \emph{principio de casillas para ordenes lineales}.

\begin{definition}
Dados \'ordenes $\ell_1,\dots,\ell_n$, escribimos $ L\to(\ell_1,\dots,\ell_n)^1 $ 
para indicar que, no importa como particionemos $L$ en $n$ piezas, $L=A_1\cup\dots\cup A_n$, 
para alg\'un $i$, $A_i$ contiene una copia de $\ell_i$ como suborden.
\end{definition}

El principio de casillas usual es el estudio de esta relaci\'on cuando los \'ordenes son finitos. El caso 
m\'as b\'asico es $n+1\to(\underbrace{2,\dots,2}_n)^1$: si $n+1$ objectos se dividen en $n$ grupos, 
al menos dos objetos terminan en el mismo grupo. 

En general, estudiar problemas de Ramsey requiere considerar problemas de casillas. En el caso 
del ejemplo \ref{ex:fin}, usamos que $5\to(3,3)^1$ para concluir que $6\to(3,3)^2$.

En esta notaci\'on, la pregunta \ref{q:Z} es qu\'e tan largo ha de ser $L$ para asegurar que  
$L\to(\mathbb Z,\mathbb Z)^1$ o, m\'as precisamente: ?`qu\'e orden $\mathcal L$ es tal que 
$ L\to(\mathbb Z,\mathbb Z)^1$ si y s\'olo si $L\to(\mathcal L)^1$ (es decir, si y s\'olo si $L$ contiene 
una copia de $\mathcal L$)? (Veremos que necesitamos una peque\~na modificaci\'on de esta 
versi\'on de la pregunta.)

\begin{definition}
Un orden $L$ es \emph{indescomponible} si y s\'olo si $L\to(L,L)^1$, es decir, si y s\'olo si en 
cualquier partici\'on de $L$ en dos piezas, por lo menos una de las piezas contiene de nuevo una 
copia de $L$.
\end{definition}

Por ejemplo, $\omega,\omega^*$ son indescomponibles, pero $\mathbb Z$ no lo es.

\begin{lemma}
Si $A$ y $B$ son \'ordenes indescomponibles, tambi\'en lo es $A\cdot B$.
\end{lemma}

Se sigue que, por ejemplo, $\omega\cdot\omega^*$, $\omega^2=\omega\cdot\omega$, 
$\omega^*\cdot\omega$, y $(\omega^*)^2=\omega^*\cdot\omega^*=(\omega^2)^*$ son 
indescomponibles.

\begin{proof}
Escribamos $A \cdot B=\bigcup_{b\in B}A_b$, donde cada $A_b$ es una copia de $A$, y los puntos 
en $A_b$ vienen despu\'es de los puntos en $A_{b'}$ si y s\'olo si $b'<b$.

Consideremos una partici\'on de $A\cdot B$, digamos $C\cup D$. Esto induce una partici\'on de 
$B$ en dos conjuntos, $E$ y $B\smallsetminus E$, donde $ E=\{b: A_b\cap C\to(A)^1\}$. 

Como $A$ es indescomponible y $(A_b\cap C)\cup(A_b\cap D)=A_b$, si 
$A_b\cap C\nrightarrow (A)^1$, entonces $A_b\cap D\to(A)^1$.

Hay dos casos: O bien (1) $E\to(B)^1$, o (2) no.

En caso (1), $C\to(A\cdot B)^1$.

En caso (2), $B\smallsetminus E\to(B)^1$ porque $B$ es indescomponible. Como indicamos arriba, 
$A_b\cap D\to(A)^1$ para cada $b\in B\smallsetminus E$, y por tanto  
$D\to(A\cdot B)^1$.
\end{proof}

Podemos ahora responder la pregunta \ref{q:Z}.

\begin{theorem} \label{thm:partZ}
Los siguientes son equivalentes para un orden $L$:
\begin{enumerate}
\item
$L\to(\mathbb Z,\mathbb Z)^1$,
\item 
$L\to(\omega,\omega^*)^1$, y
\item 
$L\to(\omega\cdot\omega^*\mbox{ o }\omega^*\cdot\omega)^1$, es decir, $L$ contiene un suborden 
de uno de estos tipos.
\end{enumerate}
\end{theorem}

De modo que hay precisamente dos respuestas incomparables a la pregunta de qu\'e tan largo ha 
de ser $L$.

\begin{proof}
($1.\Rightarrow 2.$) Esto es obvio, porque $\mathbb Z$ contiene sub\'ordenes de tipo $\omega$ y 
de tipo $\omega^*$.

\vspace{1mm}
($2.\Rightarrow 3.$) Sea $\mathcal C$ la clase de \'ordenes $L$ tales que 
$L\to(\omega,\omega^*)^1$. La primera observaci\'on es que $\mathcal C$ misma es, en cierto 
sentido, indescomponible:

\begin{fact}
Si $L\in\mathcal C$ y $L=A\cup B$, entonces $A$ o $B$ est\'a en $\mathcal C$.
\end{fact}

\begin{proof}
Supongamos que $L=A\cup B$ es un orden, que  
$A\nrightarrow (\omega,\omega^*)^1$, como atestigua la partici\'on $A=C\cup D$, y que  
$B\nrightarrow (\omega,\omega^*)^1$, como atestigua $B=E\cup F$.

Consideremos la partici\'on de $L$ como $(C\cup E)\cup(D\cup F)$, y notemos que  
 $$ C\cup E\nrightarrow (\omega)^1 \quad\mbox{ y }\quad D\cup F\nrightarrow(\omega^*)^1, $$ 
de modo que esta partici\'on atestigua que $L\notin\mathcal C$.
\end{proof}

Sea $L\in\mathcal C$. Hay dos casos: O bien (1) $L$ contiene un orden $L'$ tal que $L'$ y todos 
sus segmentos finales no vac\'\i os est\'an en $\mathcal C$, o (2) no.

En caso (1), como el orden $L'$ est\'a en $\mathcal C$, \'el contiene una copia de $\omega^*$, 
seguida de por lo menos un punto. Por la propiedad de $L'$, el segmento final de $L'$ que sigue a 
esta copia de $\omega^*$ de nuevo contiene una copia de $\omega^*$ y por lo menos otro punto, y 
podemos repetir. 

Se sigue que $L'$, y por tanto tambi\'en $L$, contiene una copia de $\omega^*\cdot \omega$.

En caso (2), $L$ mismo contiene un segmento final no vac\'\i o que no est\'a en $\mathcal C$. Sea  
$T$ la uni\'on de todos estos segmentos finales.

\begin{fact}
$T\in \mathcal C$.
\end{fact}

\begin{proof}
De lo contrario, $T\notin\mathcal C$ (en particular, $T\ne L$) y se sigue que todos los segmentos 
finales no vac\'\i os de $L\smallsetminus T$ est\'an en $\mathcal C$, lo que contradice que no 
estamos en el caso (1).
\end{proof}

Tenemos que $T\in\mathcal C$ pero ninguno de sus segmentos finales propios lo est\'a, de modo 
que todos sus segmentos iniciales no vac\'\i os est\'an en $\mathcal C$.

Como el orden $T$ est\'a en $\mathcal C$, \'el contiene una copia de $\omega$ precedida por al 
menos un punto. Por la propiedad de $T$, el segmento inicial que precede esta copia de $\omega$ 
de nuevo contiene una copia de $\omega$ y por lo menos un punto antes, y podemos repetir.

Se sigue que $T$, y por tanto $L$, contiene una copia de $\omega\cdot \omega^*$.

\vspace{1mm}
($3.\Rightarrow 1.$) Es claro que si $\varphi\in\{\omega\cdot\omega^*,\omega^*\cdot\omega\}$, 
entonces $\varphi\to(\mathbb Z)^1$. El resultado se sigue, porque $\omega\cdot\omega^*$ y 
$\omega^*\cdot\omega$ son indescomponibles. De hecho, esto muestra que 1. es equivalente a 
$L\to(\underbrace{\mathbb Z,\dots,\mathbb Z}_n)^1$ para cualquier $n\ge 2$.
\end{proof}

\section{Un teorema de Ramsey para los enteros} \label{section:teorema}

Podemos ahora establecer el resultado prometido.

\begin{theorem} \label{thm:Z2}
Los siguientes son equivalentes para un orden $L$ y cualquier $n\ge3$:
\begin{enumerate}
\item 
$L\to(\omega\cdot \omega^*\mbox{ o }\omega^*\cdot\omega)^1$,
\item
$L\to(\mathbb Z,n)^2$.
\end{enumerate}
\end{theorem}

Es decir, que la generalizaci\'on del n\'umero de Ramsey (el \emph{orden} de Ramsey) 
$R(\mathbb Z,n)$ en vez de ser un orden es un conjunto finito de dos \'ordenes incomparables.

\begin{proof}
($1.\Rightarrow 2.$) Por el teorema \ref{thm:specker}, $\omega^2\to(\omega^2,n)^2$ para todo 
$n\ge2$. Se sigue que $\omega^*\cdot \omega\to(\omega^*\cdot\omega,n)^2$ para todo 
$n\ge2$, simplemente identificando la $n$-\'esima copia de $-\mathbb N$ en 
$\omega^*\cdot\omega$ con la $n$-\'esima copia de $\mathbb N$ en $\omega^2$ de la manera 
natural. Pero $\omega^*\cdot\omega$ contiene una copia de $\mathbb Z$.  El otro caso es 
an\'alogo.

\vspace{1mm}
($2.\Rightarrow 1.$) Usamos el teorema \ref{thm:partZ}. Supongamos que 
$L\nrightarrow(\mathbb Z,\mathbb Z)^1$, como atestigua la partici\'on $L=A\cup B$. Coloreemos 
$K_L$ de modo que cualquier arco entre dos puntos de $A$ o entre dos puntos de $B$ es 
rojo, y entre un punto de $A$ y uno de $B$ es azul. Este coloreo atestigua que 
$L\nrightarrow(\mathbb Z,3)^2$: el coloreo no admite tri\'angulos azules, porque de cualesquiera 
tres puntos, dos pertenecen a la misma parte de la partici\'on, y el arco entre ellos es por tanto rojo, 
y para cualquier conjunto $H\subseteq L$ con $K_H$ rojo, todos los v\'ertices de $H$ pertenecen 
a la misma parte de la partici\'on, y por tanto $H$ no contiene una copia de $\mathbb Z$.   
\end{proof}

\section{Coda} \label{section:coda}

Aunque el teorema \ref{thm:Z2} y nuestras observaciones anteriores cubren todos los posibles 
resultados de la forma $L\to(\mathbb Z,L_2)^2$, hay una pregunta adicional: ?`podemos encontrar 
un orden $L$ tal que cualquier coloreo de  $K_L$  o bien contiene una copia 
roja de $K_{\mathbb Z}$, o hay un conjunto infinito $X$ tal que $K_X$ es azul? 

La existencia de tal orden no est\'a prohibida por el teorema \ref{thm:ER}, porque para algunos 
coloreos podr\'\i a ser que $X$ sea de tipo $\omega$ y para otros de tipo $\omega^*$. 

Resulta que, de hecho, una caracterizaci\'on de tales \'ordenes, por lo menos en el caso contable, 
ya aparece en la literatura. La relaci\'on correspondiente se escribe en s\'\i mbolos 
$L\to(\mathbb Z,\aleph_0)^2$. Por completitud menciono los resultados relevantes.

\begin{theorem}[Erd\H os y Rado {\cite[Theorem 4]{MR65615}}]
El conjunto $\mathbb Q$ de los racionales satisface $\mathbb Q\to(\mathbb Q,\aleph_0)^2$.
\end{theorem}

En particular, $\mathbb Q\to(\mathbb Z,\aleph_0)^2$. Los \'ordenes que contienen una copia de 
$\mathbb Q$ se denominan no dispersos. 

\begin{theorem}[Erd\H os y Hajnal {\cite[Corollary 1]{MR141602}}]
Si $L\to(\mathbb Z,\aleph_0)^2$ y $L$ es contable, entonces $L$ es no disperso.
\end{theorem}

Espero que los resultados mencionados en esta nota, y la historia acompa\~nante, ayuden a ilustrar 
la descripci\'on intuitiva dada por Motzkin \cite{MR214478} de la teor\'\i a de Ramsey: el desorden 
completo es imposible.

{\small\subsection*{Agradecimientos}
El autor le agradece a Ramiro de la Vega por conversaciones en las etapas iniciales de este 
proyecto, y a Ramiro, David Fern\'andez y Sharif Vel\'asquez por comentarios sobre esta nota. 
Gracias a Alexander Cardona y a los editores por la invitaci\'on a publicar.}

\bibliographystyle{amsalpha}
\bibliography{zramsey}

\end{document}